\documentclass[12pt]{amsart}

\usepackage{amssymb}
\usepackage{amscd}

\title{Bounds for the number of Galois points for plane curves}
\author{Satoru Fukasawa}

\address{Department of Mathematical Sciences, Faculty of Science, Yamagata University, Kojirakawa-machi 1-4-12, Yamagata 990-8560, Japan}
\email{s.fukasawa@sci.kj.yamagata-u.ac.jp} 
\thanks{The author was partially supported by JSPS KAKENHI Grant Number 25800002.}

\newtheorem{theorem}{Theorem}[section]
\newtheorem{proposition}[theorem]{Proposition} 
\newtheorem{corollary}[theorem]{Corollary}
\newtheorem{lemma}[theorem]{Lemma}
\newtheorem{fact}[theorem]{Fact}

\theoremstyle{definition}
 
\newtheorem{remark}[theorem]{Remark}

\begin{document}
\begin{abstract} 
A point on a plane curve is said to be Galois (for the curve) if the projection from the point as a map from the curve to a line induces a Galois extension of function fields. 
It is known that the number of Galois points is finite except for a certain explicit example. 
We establish upper bounds for the number of Galois points for all plane curves other than the example in terms of the genus, degree and the generic order of contact, and settle curves attaining the bounds. 
\end{abstract}
\maketitle

\section{Introduction}  
In 1996, H. Yoshihara introduced the notion of {\it Galois point} (\cite{miura-yoshihara, yoshihara1}). 
Let $C \subset \Bbb P^2$ be an irreducible plane curve of degree $d \ge 4$ over an algebraically closed field $K$ of characteristic $p \ge 0$ and let $K(C)$ be its function field. 
A point $P \in C$ is said to be (inner) Galois for $C$, if the function field extension $K(C)/\pi_P^*K(\Bbb P^1)$ induced by the projection $\pi_P: C \dashrightarrow \Bbb P^1$ from $P$ is Galois. 
We denote by $\delta(C)$ (resp. $\delta_s(C)$) the number of Galois points in the smooth (resp. singular) locus of $C$. 
It would be interesting to settle $\delta(C)$ or $\delta(C)+\delta_s(C)$.  
For example, there are applications of the distribution of Galois points to finite geometry (see \cite{fhk, oyono-ritzenthaler}). 

When $p=0$, Yoshihara and K. Miura \cite{miura-yoshihara, yoshihara1} showed that $\delta(C)=0, 1$ or $4$ for smooth curves. 
In $p>0$, for the Fermat curve $H$ of degree $p^e+1$, M. Homma \cite{homma2} proved that $\delta(H)=(p^e)^3+1$. 
Recently, the present author \cite{fukasawa1} showed that $\delta(C)=0, 1$ or $d$ for any other smooth curve $C$. 
As a next step, it would be nice to {\it give an upper bound for $\delta(C)$} for all irreducible plane curves $C$. 
Miura \cite{miura} gave a certain inequality related to $\delta(C)$ if $p=0$ and $d-1$ is prime. 
The present author and T. Hasegawa \cite{fukasawa-hasegawa} settled the case $\delta(C)=\infty$. 
We call this case (FH).  
In \cite{fukasawa3}, the present author showed that $\delta(C) \le (d-1)^3+1$ for all irreducible plane curves $C$ if $p \ne 2$ and $C$ is not in the case (FH). 

Let $M(C)$ be the generic order of contact, i.e. $I_P(C, T_PC)=M(C)$ for a general point $P \in C$.  
In this paper, we establish a sharper bound in arbitrary characteristic, as follows. 

\begin{theorem} \label{main theorem}
Let $C \subset \Bbb P^2$ be an irreducible plane curve of degree $d \ge 4$ in characteristic $p \ge 0$ and let $g$ be the geometric genus. 
Assume that $C$ is not in the case (FH). 
Then, 
$$\delta(C) \le (M(C)+1)(2g-2)+3d. $$ 
Furthermore, the equality holds if and only if $p>0$, $d=p^e+1$ for some $e>0$ and $C$ is projectively equivalent to one of the following plane curves:  
\begin{itemize}
\item[(1)] Fermat curve. 
\item[(2)] The image of the morphism 
$$ \Bbb P^1 \rightarrow \Bbb P^2; \ (s:t) \mapsto (s^{p^e+1}:(s+t)^{p^e+1}:t^{p^e+1}). $$ 
\end{itemize}
\end{theorem}

It follows from \cite{fukasawa2} that, for curves in Theorem \ref{main theorem}(2), $\delta(C)=d=(p^e+1)(-2)+3d$ and $\delta_s(C)=(d-1)(d-2)/2$. 
Since the number of singular points of $C$ is at most $(d-1)(d-2)/2-g$ by genus formula (\cite[V. Example 3.9.2]{hartshorne}, \cite[p.135]{hkt}), we have the following. 

\begin{corollary} 
Under the same assumption as in Theorem \ref{main theorem}, 
$$\delta(C) +\delta_s(C) \le (M(C)+1)(2g-2)+3d+\frac{(d-1)(d-2)}{2}-g. $$ 
Furthermore, the equality holds if and only if the same condition as in Theorem \ref{main theorem} holds. 
\end{corollary}

In $p=0$, it is well known that $M(C)=2$ (see, for example, \cite{kleiman1, pardini}). 
Therefore, we have the following. 

\begin{corollary}
Let $C \subset \Bbb P^2$ be an irreducible plane curve of degree $d \ge 4$ in characteristic $p=0$ and let $g$ be the geometric genus. 
Then, 
$$\delta(C) < 3(2g-2)+3d. $$ 

\end{corollary}

Nowadays six types of plane curves $C$ with $\delta(C) \ge 2$ are known (Table in \cite{open}). 
Our results imply that the three of them ((FH) and curves in Theorem \ref{main theorem}(1)(2)) are characterized by the number $\delta(C)$ or $\delta(C)+\delta_s(C)$. 

Our results can be considered as an application of projective geometry in positive characteristic (including the theories of projective duality and Weierstrass points, Pl\"ucker formula and Kaji's theorem, see \cite{hefez-kleiman, kaji, kleiman1, kleiman2, piene, stohr-voloch}).

\section{Preliminaries}
Let $(X:Y:Z)$ be a system of homogeneous coordinates of the projective plane $\Bbb P^2$ and let $C_{\rm sm}$ be the smooth locus of a plane curve $C$.  
When $P \in C_{\rm sm}$, we denote by $T_PC \subset \Bbb P^2$ the (projective) tangent line at $P$. 
For a projective line $\ell \subset \Bbb P^2$ and a point $P \in C \cap \ell$, $I_P(C, \ell)$ means the intersection multiplicity of $C$ and $\ell$ at $P$.  
We denote by $\overline{PR}$ the line passing through points $P$ and $R$ when $R \ne P$, and by $\pi_P$ the projection from a point $P \in \Bbb P^2$. 
The projection $\pi_P$ is represented by $R \mapsto \overline{PR}$.  
Let $r: \hat{C} \rightarrow C$ be the normalization and let $g$ be the genus of $\hat{C}$.  
We write $\hat{\pi}_P=\pi_P \circ r$. 
We denote by $e_{\hat{R}}$ the ramification index of $\hat{\pi}_P$ at $\hat{R} \in \hat{C}$. 
If $R=r(\hat{R}) \in C_{\rm sm}$, then we denote $e_{\hat{R}}$ also by $e_R$.   
It is not difficult to check the following.  

\begin{lemma} \label{index}
Let $P \in \Bbb P^2$ and let $\hat{R} \in \hat{C}$ with $r(\hat{R})=R \ne P$. 
Then for $\hat{\pi}_P$ we have the following. 
\begin{itemize}
\item[(1)] If $P \in C_{\rm sm}$, then $e_P=I_P(C, T_PC)-1$.  
\item[(2)] If $h$ is a linear polynomial defining $\overline{RP}$, then $e_{\hat{R}}={\rm ord}_{\hat{R}}r^*h$. 
In particular, if $R$ is smooth, then $e_R =I_R(C, \overline{PR})$.  
\end{itemize} 
\end{lemma}

Let $\check{\Bbb P}^2$ be the dual projective plane. 
The dual map $\gamma: C_{\rm sm} \rightarrow \check{\Bbb P}^2$ is given by $P \mapsto T_PC$ and the dual curve $C^* \subset \check{\Bbb P}^2$ is the closure of the image of $\gamma$. 
We denote by $q(\gamma)$ (resp. $s(\gamma)$) the inseparable (resp. separable) degree of the field extension induced by the dual map $\gamma$ of $C$ onto $C^*$. 
It is known that $q(\gamma)=1$ if and only if $p \ne 2$ and $M(C)=2$, and $s(\gamma)=1$ in this case (see, for example, \cite[Proposition 1.5]{pardini}). 
If $q(\gamma) \ge 2$, then it follows from a theorem of Hefez and Kleiman (\cite[(3.4)]{hefez-kleiman}) that $M(C)=q(\gamma)$. 
For strange curves, see \cite{bayer-hefez, kleiman1}. 
Note that the projection $\hat{\pi}_Q$ from a point $Q$ is not separable if and only if $C$ is strange and $Q$ is the strange center.

The order sequence of the morphism $r:\hat{C} \rightarrow \Bbb P^2$ is $\{0, 1, M(C)\}$ (see \cite[Ch. 7]{hkt}, \cite{stohr-voloch}). 
If $\hat{R} \in \hat{C}$ is a non-singular branch, i.e. there exists a line defined by $h=0$ with ${\rm ord}_{\hat{R}}r^*h=1$, then there exists a unique tangent line at $R=r(\hat{R})$ defined by $h_{\hat{R}}=0$ such that ${\rm ord}_{\hat{R}}r^*h_{\hat{R}} \ge M(C)$.  
We denote by $T_{\hat{R}}C \subset \Bbb P^2$ this tangent line, and by $\nu_{\hat{R}}$ the order ${\rm ord}_{\hat{R}}r^*h_{\hat{R}}$ of the tangent line $h_{\hat{R}}=0$ at $\hat{R}$. 
If $\nu_{\hat{R}}-M(C)>0$, then we call the point $\hat{R}$ (or $R=r(\hat{R})$ if $R \in C_{\rm sm}$) a {\it flex}.  
We denote by $\hat{C}_0 \subset \hat{C}$ the set of all non-singular branches and by $F(\hat{C}) \subset \hat{C}_0$ the set of all flexes.  
We recall the following two facts (see \cite[Theorem 1.5]{stohr-voloch}, \cite{piene}).  

\begin{fact}[Count of flexes] \label{flexes} We have
$$ \sum_{\hat{R} \in \hat{C}_0} (\nu_{\hat{R}}-M(C)) \le (M(C)+1)(2g-2)+3d. $$
\end{fact}

\begin{fact}[Pl\"ucker formula] \label{plucker} 
Let $d^*$ be the degree of the dual curve $C^*$. 
Then
$$ s(\gamma)q(\gamma)d^* \le 2g-2+2d. $$
Furthermore, if $\hat{C}_0=\hat{C}$ (i.e. $r: \hat{C} \rightarrow \Bbb P^2$ is unramified), then the equality holds. 
\end{fact}

To settle curves attaining our bound, we also use an important {\it theorem of Kaji} \cite{kaji, kleiman2} on (non-reflexive) curves, which asserts that the geometric genera of $C$ and $C^*$ are the same. 

On a Galois covering of curves, the following holds in general (see \cite[III. 7.2]{stichtenoth}). 

\begin{fact} \label{Galois covering} 
Let $\theta: C \rightarrow C'$ be a Galois covering of degree $d$. 
We denote by $e_P$ the ramification index at a point $P \in C$. 
Then we have the following. 
\begin{itemize}    
\item[(1)] If $\theta(P)=\theta(Q)$, then $e_P=e_Q$. 
\item[(2)] The index $e_P$ divides the degree $d$. 
\end{itemize} 
\end{fact}

We mention properties of Galois covering between rational curves. 
The following fact is a corollary of the classification of finite subgroups of ${\rm PGL}(2, K)$ (see, for example, \cite[Theorem 11.91]{hkt}). 

\begin{fact} \label{rational Galois covering} 
Let $\theta: \Bbb P^1 \rightarrow \Bbb P^1$ be a Galois covering of degree $d \ge 3$ and let $d=p^ek$, where $k$ is not divisible by $p$. 
If $p=0$, we set $e=0$. 
We have the following. 
\begin{itemize}
\item[(1)] Assume that $p>0$. If the index $e_P$ is equal to $p^a \ge 3$ for some $a>0$ at a point $P$, then we have $e \ge 1$. 
Further, if $d > p^a$, then $a=e$ and there exists a point $Q$ such that $e_Q \ge 2$ and $e_Q$ divides $k$. 
\item[(2)] If $e=0$ and $e_P=d$ for some $P$, then there exist exactly two points with index $d$. 
\item[(3)] If $e \ge 1$, $k \ge 2$ and $e_P=d$ for some $P$, then there exist exactly $p^e$ points with index $k$ in the same fiber.  
\item[(4)] Let $P \in \Bbb P^1$. 
If $\theta^{-1}(\theta(P))$ contains two or more points, then there exist two different points $P', P''$ such that $P', P''\not\in\theta^{-1}(\theta(P))$ and $e_{P'}, e_{P''} \ge 2$. 
Furthermore, if all such pairs $P', P''$ satisfies $e_{P'}=e_{P''}=2$, then the Galois group is the dihedral group and $\theta^{-1}(\theta(P))$ consists of exactly two points. 
\item[(5)] Assume that $p=2$. 
If there exists no point $P$ with $e_P \ge 4$, then there exist two different points $P'$, $P''$ such that $e_{P'} \ge 3$, $e_{P''} \ge 3$. 
\end{itemize} 
\end{fact} 

\section{Preliminary propositions on  Galois points} 
In \cite[Lemma 2.5]{fukasawa3}, the present author proved the following. 

\begin{lemma}[\cite{fukasawa3}] \label{twoGalois} 
Let $P_1, P_2 \in C_{\rm sm}$ be two distinct Galois points and let $h$ be a defining polynomial of the line $\overline{P_1P_2}$. 
Then, ${\rm ord}_{\hat{R}} r^*h=1$ for any $\hat{R} \in \hat{C}$ with $R=r(\hat{R}) \in \overline{P_1P_2}$ (maybe $R=P_1$ or $P_2$). 
\end{lemma}

We denote by $\Delta \subset \hat{C}$ the set of all points $\hat{P} \in \hat{C}$ such that $r(\hat{P}) \in C$ is smooth and Galois for $C$.  
If $M(C)$ is large, we have the following. 

\begin{proposition} \label{high order}
Assume that $(d+2)/2 \le M(C) \le d-1$. 
Then, $\Delta \subset F(\hat{C})$ and $\delta(C)(d-M(C)) \le (M(C)+1)(2g-2)+3d$. 
In particular, if $g=0$, then $\delta(C) \le d$. 
\end{proposition} 

\begin{proof}
Let $\hat{P} \in \Delta(C)$ and let $P=r(\hat{P}) \in C_{\rm sm}$. 
We prove that $C \cap T_PC=\{P\}$. 
Assume by contradiction that $Q \in C \cap T_PC$ and $Q \ne P$. 
It follows from Lemma \ref{index} and Fact \ref{Galois covering}(1) that 
$$ d \ge M(C)+(M(C)-1) \ge d+1. $$
This is a contradiction. 
Therefore, $I_P(C, T_PC)=d$ and $\hat{P} \in F(\hat{C})$. 
It follows from Fact \ref{flexes} that 
$$ \delta(C)(d-M(C)) \le (M(C)+1)(2g-2)+3d. $$
Let $g=0$. 
Since 
\begin{eqnarray*}
& &d(d-M(C))-\{(M(C)+1)(-2)+3d\} \\
&=& d^2-3d+2+M(C)(2-d) \\
&=& (d-2)\{(d-1)-M(C)\} \ge 0, 
\end{eqnarray*}
we have $\delta(C) \le d$. 
\end{proof} 

In a very special case, we have the following proposition needed later. 

\begin{proposition} \label{cusp}
Assume that $M(C)=2$, $Q \in C$ is a (unique) singular point with multiplicity $d-1$ and $r^{-1}(Q)=\{\hat{Q}\}$.  
Then, $\delta(C) \le d-2$. 
\end{proposition}

\begin{proof} 
First we prove that $d^* \le d$. 
We may assume that $Q=(0:0:1)$ and $C$ is defined by $X^{d-1}Z-A_d(X, Y)$, where $A_d$ is a homogeneous polynomial of degree $d$ in variables $X, Y$. 
Let $y=Y/X$, $z=Z/Y$ and let $a_d=A_d(1, y)$. 
Note that $\frac{dz}{dy}=a_d'$
and $z-y\frac{dz}{dy}=a_d-y a_d'$. 
The dual map is represented by the matrix 
$$ \left(\begin{array}{ccc} 
1 & y & z \\
0 & 1 & \frac{dz}{dy} 
\end{array}\right) \sim
\left(\begin{array}{ccc}
1 & 0 & z-y\frac{dz}{dy} \\
0 & 1 & \frac{dz}{dy}
\end{array}\right)
$$
in general. 
Since the degrees of polynomials $z-y\frac{dz}{dy}$ and $\frac{dz}{dy}$ are at most $d$, we have $d^* \le d$. 

We remark on the number of flexes. 
Since ${\rm ord}_{\hat{Q}}r^*h \ge d-1$ for any line $\ell \ni Q$ defined by a linear polynomial $h=0$, by \cite[Theorem 1.5]{stohr-voloch}, we have 
\begin{eqnarray*}
\sum_{R \in C \setminus \{Q\}} (I_R(C, T_RC)-2) & \le & (2+1)(-2)+3d-\{(d-1-1)+(d-2)\} \\
&=& d-2.   
\end{eqnarray*}

Let $d-1=p^ek$, where $k$ is not divisible by $p$. 
Assume that $k \ge 3$. 
Note that $e_{\hat{Q}}=d-1$ for the projection $\hat{\pi}_P$ from any point $P \in r(\Delta)$. 
It follows from Fact \ref{rational Galois covering}(2)(3) and Lemma \ref{index} that for any $P \in r(\Delta)$ there exists a point $P' \in C_{\rm sm}$ such that $T_{P'}C \ni P$ and $I_{P'}(C, T_{P'}C)=k$ or $k+1$. 
Note that $P'$ is different for each $P$, by Lemma \ref{twoGalois}. 
Since $P'$ is a flex, $\delta(C) \le d-2$. 

Assume that $k=2$. 
Then, $e \ge 1$ and $p \ge 3$. 
Since $M(C)=2$, $s(\gamma)=1$.  
It follows from Fact \ref{rational Galois covering}(3) and Lemma \ref{index} that for any $P \in r(\Delta)$ there exist $p^e$ points $P' \in C \setminus \{Q\}$ in the same line $\ell \ni P$ such that $I_{P'}(C, \ell)=2$ or $3$, i.e. $\ell=T_{P'}C$. 
Then, $\gamma(P') \in C^*$ is a singular point of $C^*$ with multiplicity at least $p^e=(d-1)/2$, since the set $C \cap T_{P'}C$ contains $(d-1)/2$ points and the tangent lines at such points are equal to $T_{P'}C$. 
Note that $d \ge 3\times 2+1=7$. 
Since $d^* \le d$, by Lemma \ref{twoGalois} and genus formula for $C^*$ (\cite[V. Example 3.9.2]{hartshorne}, \cite[p.135]{hkt}), we have the inequality 
$$ \delta(C) \times \frac{1}{2}\frac{d-1}{2}\left(\frac{d-1}{2}-1\right) \le \frac{(d-1)(d-2)}{2}. $$
Then, we have $\delta(C) \le d-2$. 

Similarly to \cite[p. 116, lines 6--23]{fukasawa3}, if $k=1$, then we have $\delta(C)=1$. 
\end{proof}

If $\delta(C)$ is large enough, we can prove that $C$ is an ``immersed'' curve. 
Precisely, we have the following.

\begin{proposition} \label{unramified} 
Assume that there exists no singular point with multiplicity $d-1$. 
If $2g-2+2d-2 < \delta(C)$, then $\hat{C}_0=\hat{C}$. 
\end{proposition}

\begin{proof}
Let $Q$ be a singular point with multiplicity $m \le d-2$. 
Note that the number of tangent directions at $Q$ is at most $m$. 
Assume that $Q$ is not a strange center. 
We prove that any point $\hat{R} \in \hat{C}$ with $r(\hat{R})=Q$ is a non-singular branch. 
If there exists a line containing $Q$ and two Galois points, then we have this assertion by Lemma \ref{twoGalois}. 
Therefore, we consider the case where any line containing $Q$ has at most one inner Galois point. 
We consider the projection $\hat{\pi}_{Q}$ from $Q$. 
It follows from Riemann-Hurwitz formula that the number of ramification points is at most $2g-2+2(d-m)$. 
By the assumption, $2g-2+2(d-m)+m \le 2g-2+2(d-2)+2 <\delta(C)$. 
Then, there exists a line $\ell \ni Q$ such that (the point of $\Bbb{P}^1$ corresponding to) $\ell$ is not a branch point of $\hat{\pi}_Q$, the fiber $r^{-1}(C \cap \ell \setminus \{Q\})$ consists of $d-m$ points, and $\ell$ contains a Galois point $P \in C_{\rm sm}$. 
In this case, there exists a point $\hat{R} \in \hat{C}$ with $R=r(\hat{R}) \in \overline{PQ} \setminus \{P, Q\}$ such that ${\rm ord}_{\hat{R}}r^*h=1$, where $h$ is a defining polynomial of the line $\overline{PR}$, by Lemma \ref{index}. 
It follows from Fact \ref{Galois covering}(1) that any point $\hat{R} \in r^{-1}(Q)$ is a non-singular branch.

We prove that $Q$ is not a strange center (under the assumption that any point $\hat{R}$ with $r(\hat{R}) \ne Q$ is a non-singular branch). 
Assume by contradiction that $Q$ is a strange center. 
Then, $p>0$. 
Since the projection $\hat{\pi}_{Q}$ from $Q$ is not separable, $e_P=p^ek$ for some integers $e >0$ and $k>0$ for each $P \in r(\Delta)$. 
By Lemma \ref{index}(2), $I_P(C, T_PC)=p^ek$. 
We consider the projection $\hat{\pi}_P$ from a point $P \in r(\Delta)$. 
By Lemma \ref{index}(1), $e_P=I_P(C, T_PC)-1=p^ek-1$. 
On the other hand, $Q \in T_{\hat{R}}C$ for any point $\hat{R}$ with $r(\hat{R}) \ne Q$, since any point $\hat{R}$ with $r(\hat{R}) \ne Q$ is a non-singular branch by the above discussion.  
Therefore, the projection $\hat{\pi}_P$ is ramified only at points in the line $\overline{PQ}$. 
There exist only tame ramification points for $\hat{\pi}_P$, by Fact \ref{Galois covering}(1). 
By Riemann-Hurwitz formula, this is a contradiction. 
\end{proof}

\section{The case where $M(C) \ge 3$} 
Throughout this section, we assume that $(M(C)+1)(2g-2)+3d \le \delta(C) < \infty$. 
If $M(C)=d$, then the present author showed that $\delta(C)=0$ or $\infty$ in 
\cite{fukasawa2} (Classification of curves with $M(C)=d$ by Homma \cite[Theorem 3.4]{homma1} is crucial). 
Therefore, we have $M(C)<d$. 
We prove that (i) $\hat{C}_0=\hat{C}$ if $g \ge 1$, and (ii) $\Delta \subset F(\hat{C})$.

We consider the case where there exists a singular point $Q$ with multiplicity $d-1$. 
Then, $\hat{C}$ is rational and $Q$ is a unique singular point. 
It follows from B\'ezout's theorem that $Q \not\in T_PC$ for any point $P \in C_{\rm sm}$. 
Since $d>M(C)$, there exists a point $R \in C_{\rm sm} \setminus \{P\}$ with $R \in T_PC$ if $I_P(C, T_PC)=M(C)$. If there exists a point $P \in r(\Delta)$ with $I_P(C, T_PC)=M(C)$, then it follows from Lemma \ref{index} and Fact \ref{Galois covering}(1) that $I_R(C, T_RC)=M(C)-1 \ge 2$. 
This is a contradiction to the order sequence $\{0, 1, M(C)\}$. 
Therefore, $\Delta \subset F(\hat{C})$. 

We consider the case where there exists no singular point with multiplicity $d-1$. 
By Proposition \ref{high order}, if $g=0$ and $M(C) \ge \frac{d+2}{2}$, then $\Delta \subset F(\hat{C})$.  
Assume that $g \ge 1$ or $M(C) < \frac{d+2}{2}$. 
Then, $2g-2+2d-2 < (M(C)+1)(2g-2)+3d \le \delta(C)$. 
It follows from Proposition \ref{unramified} that $\hat{C}_0=\hat{C}$. 
We have assertion (i) above. 

Since $d>M(C)$, if there exists a Galois point $P \in C_{\rm sm}$ such that $I_P(C, T_PC)=M(C)$, then there exists a point $R \in C \cap T_PC$. 
Then, the ramification index $e_{\hat{R}}=M(C)-1 \ge 2$ at $\hat{R}$ for $\hat{\pi}_P$, where $\hat{R} \in \hat{C}$ with $r(\hat{R})=R$. 
By $\hat{C}_0=\hat{C}$ as above, this is a contradiction to the order sequence $\{0, 1, M(C)\}$.  
We have assertion (ii). 
By the assumption $\delta(C) \ge (M(C)+1)(2g-2)+3d$ and Fact \ref{flexes}, $\Delta=F(\hat{C})$ and $\delta(C)=(M(C)+1)(2g-2)+3d$. 

We consider curves with $\Delta=F(\hat{C})$ and $\delta(C)=(M(C)+1)(2g-2)+3d$. 
For each $P \in r(\Delta)$, $I_P(C, T_PC)=M(C)+1$. 
By Lemma \ref{index}(1) and Fact \ref{Galois covering}(2), $M(C)$ divides $d-1$.  
If $g \ge 1$, then $\hat{C}_0=\hat{C}$. 
By Fact \ref{plucker}, $M(C)$ divides $2g-2+2d$. 
Then, $2g \equiv 0$ modulo $M(C)$.  
When $g=1$, $M(C) \ge 3$ does not divide $2g=2$.   
We have $g \ne 1$. 

We consider the case $g \ge 2$. 
Assume that $d-1>M(C)$. 
According to Kaji's theorem \cite[Theorem 4.1, Corollary 4.4]{kaji}, the geometric genus of $C^*$ is equal to that of $C$ and $s(\gamma)=1$. 
For each $P \in r(\Delta)$, $\gamma(P) \in C^*$ is a singular point of $C^*$ with multiplicity at least $(d-1)/M(C)$, since the set $r^{-1}(C \cap T_PC)$ consists of $(d-1)/M(C)$ points and the tangent line at each point of $r^{-1}(C \cap T_PC)$ coincides with $T_PC$. 
We take $M:=M(C)$. 
By genus formula for $C^*$ (\cite[V. Example 3.9.2]{hartshorne}, \cite[p.135]{hkt}) and Fact \ref{plucker}, we have 
\begin{eqnarray*} 
g& \le& \frac{1}{2}\left(\frac{2g-2+2d}{M}-1\right)\left(\frac{2g-2+2d}{M}-2\right) \\ 
& & -((M+1)(2g-2)+3d)\times \frac{1}{2}\frac{d-1}{M}\left(\frac{d-1}{M}-1\right). \end{eqnarray*} 
Then, 
\begin{eqnarray*} 
0& \le& (2g)\{2g+(4(d-1)-3M)-(M+1)(d-1)(d-1-M)-M^2\} \\
& & +(d-1-M)\{4(d-1)-2M-(3d-2(M+1))(d-1)\}.  
\end{eqnarray*}
Since $(d-1-M)\{4(d-1)-2M-(3d-2(M+1))(d-1)\} <0$, we have 
$$ 2g \ge -(4(d-1)-3M)+(M+1)(d-1)(d-1-M)+M^2. $$
By using the inequality $2g \le (d-1)(d-2)$, 
$$ Md^2-(M^2+3M+3)d+2M^2+5M+3 \le 0. $$
We have $d < M+3+3/M \le M+4$. 
Since $d \ge 2M+1$, we have $M <3$. 
This is a contradiction. 
We have $d-1=M(C)$. 
According to classification of curves with $M(C)=d-1$ by Ballico and Hefez \cite{ballico-hefez}, $C$ is projectively equivalent to the Fermat curve. 

We consider the case $g=0$. 
If $d-1>M(C)$, it follows from Fact \ref{rational Galois covering}(1) and Lemma \ref{twoGalois} that there exists a point $\hat{Q} \not\in \Delta=F(\hat{C})$ such that $e_{\hat{Q}}$ divides $(d-1)/M(C)$ which is not divisible by $p$. 
By Lemma \ref{index}(2), $M(C)=e_{\hat{Q}}$.  
This is a contradiction. 
Therefore, we have $d-1=M(C)$. 
According to \cite{fukasawa2} (Classification of curves with $M(C)=d-1$ by Ballico and Hefez \cite{ballico-hefez} is crucial), $C$ is projectively equivalent to the image of the morphism  
$$ \Bbb P^1 \rightarrow \Bbb P^2; \  (s:t) \mapsto (s^{p^e+1}:(s+t)^{p^e+1}:t^{p^e+1}). $$

\begin{remark}
The condition $\Delta \subset F(\hat{C})$ does not hold in general (see \cite[Example 1]{miura}). 
\end{remark} 

\section{The case where $p \ne 2$ and $M(C)=2$}
{\bf Strategy for the case $M(C)=2$.}
In Sections 5 and 6, we assume by contradiction that $3(2g-2)+3d \le \delta(C) < \infty$.
To achieve a contradiction, we use upper bounds of the number of flexes: $3(2g-2)+3d$, and the number of singular points of the dual curve $C^*$: $(d^*-1)(d^*-2)/2-g$. 
For example, if there exist two different points $\hat{Q}_1$, $\hat{Q}_2 \in \hat{C}_0$ with $e_{\hat{Q}_i} \ge 3$ (equivalently, ${\rm ord}_{\hat{Q}_i}r^*h \ge 3$ or ${\rm ord}_{\hat{Q}_i}r^*h \ge 4$ if $r(\hat{Q}_i)=P$, where $h$ is a linear polynomial defining the line $\overline{Pr(\hat{Q}_i)}$ or $T_PC$ if $r(\hat{Q}_i)=P$ by Lemma \ref{index}) for each Galois point $P$, then the proof finishes by Lemma \ref{twoGalois} and the upper bound $3(2g-2)+3d$. 
Later, we say roughly that {\it a Galois point $P=r(\hat{P})$ needs $i$ flexes}, if there exist $j$ ($1 \le j \le i$) different points $\hat{Q}_1$, \ldots, $\hat{Q}_j \in \hat{C}_0$ with $e_{\hat{Q}_k} \ge 3$ and $\sum_{k=1}^j (e_{\hat{Q}_k}-2) = i$ for the projection $\hat{\pi}_P$. 
Note that if $P \in r(\Delta)$ and a point $\hat{Q} \in \hat{C}$ satisfies $e_{\hat{Q}}=2$ for $\hat{\pi}_P$, then $e_{\hat{R}}=2$ for any point $\hat{R} \in \hat{\pi}_P^{-1}(\hat{\pi}_P(\hat{Q}))$, by Fact \ref{Galois covering}(1). 
When $s(\gamma)=1$, $\hat{C}_0=\hat{C}$ and there exist a Galois point $P$ and a line $\ell_P \ni P$ such that $e_{\hat{Q}}=2$ for each $\hat{Q} \in \hat{C}$ contained in the fiber for $\hat{\pi}_P$ corresponding to $\ell_P$, by Lemma \ref{index}, the line $\ell_P$ causes a singular point of $C^*$ with multiplicity at least $(d-1)/2$. 
By Lemma \ref{twoGalois}, $\ell_{P_1} \ne \ell_{P_2}$ if $P_1, P_2 \in r(\Delta)$ and $P_1 \ne P_2$. 
Then, we use the upper bound $(d^*-1)(d^*-2)/2-g$. 

\

In this section we assume that $p \ne 2$ and $M(C)=2$. 
Then, $q(\gamma)=1$ and $s(\gamma)=1$. 

\subsection{The case where there exists a singular point $Q$ with multiplicity $d-1$} 
In this case, $\hat{C}$ is rational, $Q$ is a unique singular point and $3d-6 \le \delta(C)$. 
It follows from Proposition \ref{cusp} that we may assume that the number of points of $r^{-1}(Q)$ is at least two.  
Let $P \in r(\Delta)$. 
It follows from Fact \ref{rational Galois covering}(4) that there exist two different points $P', P'' \in C_{\rm sm}$ such that $e_{P'}, e_{P''} \ge 2$. 
If for all $P$ there exists a pair $P', P''$ such that $e_{P'}, e_{P''} \ge 3$, then we need at least two flexes for each Galois point. 
This is a contradiction. 
Therefore, there exists a Galois point $P$ such that all pairs $P', P''$ satisfy $e_{P'}=e_{P''}=2$. 
By Fact \ref{rational Galois covering}(4), the Galois group $G_P$ is the dihedral group and $r^{-1}(Q)$ consists of exactly two points $\hat{Q}_1, \hat{Q}_2$. 
Then, for any Galois point $P \in r(\Delta)$ the Galois group $G_{P}$ is the dihedral group. 
Note that the cardinality of the set $\{ \sigma \in {\rm Aut}(\Bbb P^1) \ | \ \sigma(\hat{Q}_1)=\hat{Q}_1, \ \sigma(\hat{Q}_2)=\hat{Q}_2 \}$ is $(d-1)/2$. 
This implies that $\delta(C)=1$. 
This is a contradiction.

\subsection{The case where there exists No singular point with multiplicity $d-1$} 
Since $(2g-2)+2d-2<3(2g-2)+3d$, by Proposition \ref{unramified}, we have $\hat{C}_0=\hat{C}$.  
It follows from Fact \ref{plucker} that $d^* \le 2g-2+2d$.

We consider the case where $d-1$ is odd. 
By Fact \ref{Galois covering}(2), for each $\hat{P} \in \Delta$, we need at least one flex. 
Using Fact \ref{Galois covering}(1), if $d-1 > 3$, then we need at least two flexes. 
Then, we have $\delta(C) < 3(2g-2)+3d$. 
Therefore, we have $d-1=3$. 
If $p$ does not divide $d-1=3$, by Riemann-Hurwitz formula, then we have at least two flexes for each Galois point. 
This is a contradiction. 
If $p$ divides $d-1$, i.e. $p=3$, then $\binom{3}{2}=3 \equiv 0$ modulo $p=3$. 
In this case, we spend degree at least two for each Galois point, as the degree of the Wronskian divisor (\cite[Theorem 1.5]{stohr-voloch}). 
This is a contradiction. 

We consider the case where $d-1$ is even. 
For each Galois point $P$ for which all ramification indices are one or two, it follows from Riemann-Hurwitz formula for the projection $\hat{\pi}_P$ that we have 
$$ (2g-2+2(d-1))\times \frac{2}{d-1}$$
singular points of $C^*$ with multiplicities at least $\frac{d-1}{2}$ (see ``Strategy''). 
Let $\alpha$ be the number of such Galois points. 
By genus formula, we have the inequality 
$$ \alpha \times \frac{1}{2}\frac{d-1}{2}\left(\frac{d-1}{2}-1\right) \times (2g-2+2(d-1))\times\frac{2}{d-1} \le \frac{(d^*-1)(d^*-2)}{2}-g. $$
Since 
$$ \frac{(d^*-1)(d^*-2)}{2}-g \le \frac{(2g-2+2d-1)(2g-2+2d-2)}{2}, $$
we have 
$$ \alpha \le (2g-2+2d-1)\times \frac{2}{d-3}.$$
The number $\beta$ of Galois points needing at least two flexes is at most $(3(2g-2)+3d)/2$. 
Since $\alpha+\beta \ge 3(2g-2)+3d$, 
$$ (2g-2+2d-1)\times\frac{2}{d-3} \ge \frac{3(2g-2)+3d}{2}. $$
Then, we have 
$$ (6d-26)g+(3d^2-23d+30) \le 0. $$
This implies $d \le 5$. 
Since $d$ is odd, we have $d=5$.  

Let $d=5$. 
Then, $p$ does not divide $d-1=4$ and we have 
$$\alpha \le (2g-2+2\times 5-1) \times \frac{2}{5-3}=2g+7.$$ 
Let $a, b$ be the numbers of branch points for the projection $\hat{\pi}_P$ such that each ramification indices (at ramification points in the fibers of branch points) are four, two, respectively. 
By Fact \ref{Galois covering}(2) and Riemann-Hurwitz formula, we have 
$$ 2g-2=4 \times (-2)+3a+2b. $$
Therefore, the number $a$ is even. 
By this, if a Galois point needs at least two flexes, then the Galois points need at least four flexes. 
Therefore, $\beta \le (3(2g-2)+3d)/4$. 
Since $\alpha+\beta \ge 3(2g-2)+3d$, 
$$ 2g+7 \ge \frac{3(3(2g-2)+3\times 5)}{4}. $$
Therefore, we have $g=0$. 
Then, $d^* \le 8$, $\delta(C) \ge 9$ and $(a, b)=(2, 0)$ or $(0, 3)$ for each $P \in r(\Delta)$. 
Note that the number of Galois points of type $(a,b)=(2,0)$ is at most two, since the number of flexes is at most $9$. 
Therefore, the number of Galois points of type $(a, b)=(0, 3)$ is at least $7$.
These cause at least $21$ singular points of $C^*$. 
Since the number of singular points of $C^*$ is at most $(8-1)(8-2)/2=21$ and Galois points of type $(a, b)=(2, 0)$ cause the singularity of $C^*$, this is a contradiction. 

\section{The case where $p=2$ and $M(C)=2$}

Since $p=2$, we have $q(\gamma)=2$. 
By Fact \ref{plucker}, $2s(\gamma)d^* \le (2g-2+2d)$.

\subsection{The case where there exists a singular point $Q$ with multiplicity $d-1$} 
In this case, the proof is the same as in Subsection 5.1. 

\subsection{The case where there exists No singular point with multiplicity $d-1$} 
Since $(2g-2)+2d-2<3(2g-2)+3d$, by Proposition \ref{unramified}, we have $\hat{C}_0=\hat{C}$.  
It follows from Fact \ref{plucker} that $2s(\gamma)d^*=2g-2+2d$. 
If $g=0$, it follows from Fact \ref{rational Galois covering}(5) that we need at least two flexes for each Galois points. 
Therefore, we have $g \ge 1$. 
According to Kaji's theorem \cite[Theorem 4.1, Corollary 4.4]{kaji}, the geometric genus of $C^*$ is equal to that of $C$ and $s(\gamma)=1$ if $g \ge 2$.

We consider the case where $d-1$ is odd. 
By Fact \ref{Galois covering}(2), for each $\hat{P} \in \Delta$, we need at least one flex. 
Using Fact \ref{Galois covering}(1), if $d-1 > 3$, then we need at least two flexes. 
Then, we have $\delta(C) < 3(2g-2)+3d$. 
Therefore, we have $d-1=3$. 
Since $p=2$ does not divide $d-1=3$, by Riemann-Hurwitz formula, we have at least two flexes for each Galois point. 
This is a contradiction. 

We consider the case where $d-1$ is even. 
If a Galois point $P \in r(\Delta)$ needs at least one flex, then it needs at least two flexes. 
Therefore, there exists a point $P \in r(\Delta)$ for which the fiber of each branch point for $\hat{\pi}_P$ consists of exactly $(d-1)/2$ points.  
Here, we prove that $s(\gamma)=1$. 
If $g \ge 2$, we mentioned above. 
Assume that $g=1$. 
Then $s(\gamma)$ divides $(d-1)/2$, since all fibers of the morphism $\hat{C} \rightarrow \hat{C^*}$ induced by the dual map consists of exactly $s(\gamma)$ points. 
This implies that $d \equiv 1$ modulo $s(\gamma)$.  
On the other hand, $d=(2g-2+2d)/2=s(\gamma)d^* \equiv 0$ modulo $s(\gamma)$. 
Therefore, $s(\gamma)=1$. 

Now we have $d^*=(2g-2+2d)/2$. 
Let $\alpha$ be the number of Galois points $P$ for which all ramification indices are one or two. 
By genus formula, we have the inequality 
$$ \alpha \times \frac{1}{2}\frac{d-1}{2}\left(\frac{d-1}{2}-1\right) \le \frac{(d^*-1)(d^*-2)}{2}-g. $$
Since 
$$ \frac{(d^*-1)(d^*-2)}{2}-g \le \frac{(2g+2d-4)(2g+2d-6)}{8}, $$
we have 
$$ \alpha \le \frac{(2g+2d-4)(2g+2d-6)}{(d-1)(d-3)}. $$
The number $\beta$ of Galois points needing at least two flexes is at most $(3(2g-2)+3d)/2$. 
Since $\alpha+\beta \ge 3(2g-2)+3d$, 
$$\frac{(2g+2d-4)(2g+2d-6)}{(d-1)(d-3)} \ge \frac{3(2g-2)+3d}{2}. $$
Then, we have 
$$ (2g)\left\{(2g)+(4d-10)-\frac{3}{2}(d-1)(d-3)\right\}+(d-2)(d-3)\left\{4-\frac{3}{2}(d-1)\right\} \ge 0. \ \ (*)$$
Therefore, 
$$ 2g+(4d-10)-\frac{3}{2}(d-1)(d-3) >0. $$
Using the inequality $2g \le (d-1)(d-2)$, we have 
$$ d^2-14d+25 <0. $$
This implies $d \le 11$. 

Let $d=11$. 
Coming back to the inequality $(*)$, we have
$(2g)(2g+34-120)+9\times8(4-15) \ge 0$.  
Since $2g \le (11-1)(11-2)=90$, this is a contradiction. 

Let $d=9$. 
If a Galois point needs at least two flexes, then it needs at least four flexes. 
Similarly to above, we have the inequality 
$$\frac{(2g+2d-4)(2g+2d-6)}{(d-1)(d-3)} \ge \frac{3(3(2g-2)+3d)}{4}. $$
Since $d=9$, 
$$\frac{(2g+14)(2g+12)}{8 \times 6} \ge \frac{18g+63}{4}. $$
We have 
$$ g^2-41g-147 \ge 0. $$
Since $g \le 28$, this is a contradiction. 

Let $d=7$. 
Now, we have 
$$ \alpha \le \frac{1}{24}((2g+10)(2g+8)-8g).$$
Let $\beta_1$ be the number of Galois points $P$ having a unique line $\ell \ni P$ such that the fiber corresponding to $\ell$ for $\hat{\pi}_P$ consists of two points with ramification index three. 
If $P$ is such a Galois point, since $\hat{\pi}_P$ is tamely ramified at such flexes, there exists a tangent line $\ell'$ such that the fiber corresponding to $\ell'$ for $\hat{\pi}_P$ consists of three points with index two. 
Then, the Galois point causes at least two singular points such that the multiplicity of one of them is at least three. 
Therefore, we have 
$$ \beta_1 \le \frac{1}{32}((2g+10)(2g+8)-8g). $$
If a Galois point needs at least four flexes, then it needs at least six. 
Let $\beta_2$ be the number of Galois points needing at least six flexes. 
Therefore, we have two inequalities 
$$ \alpha+\beta_1+\beta_2 \ge 3(2g-2)+21, \ \ 2\beta_1+4\beta_2 \le 3(2g-2)+21.  $$
Then, 
$$ 2\beta_1+4\{(6g+15)-(\alpha+\beta_1)\} \le 6g+15. $$
We have 
$$ 4\alpha+2\beta_1 \ge 18g+45. $$
Since 
$$4\alpha+2\beta_1 \le \left(\frac{1}{6}+\frac{1}{16}\right)\{(2g+10)(2g+8)-8g\}, $$
we have 
$$ 11\{(g+5)(g+4)-2g\} \ge 216g+540. $$
Therefore, 
$$ g(11g-139)-320 \ge 0.  $$
Since $g \le 15$, we have $g=15$ and $C$ is smooth. 
According to \cite[Theorem 3]{fukasawa1}, this is a contradiction.

Let $d=5$. 
If a Galois point needs at least two flexes, since $\binom{4}{2}=6\equiv 0$ modulo $2$, we need degree at least three as Wronskian divisor. 
Therefore, we have 
$$ \frac{(2g+6)(2g+4)}{4\times 2}-g \ge \frac{2(3(2g-2)+15)}{3}. $$
Then, 
$$ g^2-5g-6 \ge 0. $$
Since $g \le 6$, we have $g=6$ and $C$ is smooth. 
According to \cite[Theorem 3]{fukasawa1}, this is a contradiction.

\end{document}